\documentclass[reqno,11pt]{amsart}

\usepackage{amsmath,amssymb}
\usepackage{amsfonts, amscd}
\usepackage{epsfig, enumerate}
\usepackage{graphicx}
\usepackage{enumerate}
\usepackage{booktabs}
\usepackage{bbold}
\usepackage{dsfont}
\usepackage[utf8]{inputenc}
\usepackage[T1]{fontenc}
\usepackage{mathrsfs}
\usepackage[normalem]{ulem}

\usepackage{fontawesome5}
\usepackage{eso-pic}
\usepackage{charter}
\usepackage{fullpage}
\usepackage{url}

\usepackage{tikz}
\usetikzlibrary{backgrounds}
\usetikzlibrary{patterns,fadings}
\usetikzlibrary{arrows,decorations.pathmorphing}
\usetikzlibrary{calc}
\definecolor{light-gray}{gray}{0.95}
\usepackage{float}
\usepackage[colorlinks=true,linkcolor=blue,citecolor=magenta]{hyperref}

\newtheorem{theorem}{Theorem}[section]
\newtheorem{lemma}[theorem]{Lemma}
\newtheorem{proposition}[theorem]{Proposition}
\newtheorem{corollary}[theorem]{Corollary}
\newtheorem{remark}[theorem]{Remark}
\newtheorem{definition}{Definition}[section]

\numberwithin{equation}{section}

\newcommand{\one}{\mathds{1}}
\newcommand{\mc}[1]{{\mathcal #1}}

\newcommand{\bb}[1]{{\mathbb #1}}
\newcommand{\ds}[1]{{\mathds #1}}

\renewcommand{\epsilon}{\varepsilon}

\newcommand{\eps}{\varepsilon}
\newcommand{\Z}{\mathds{Z}}
\newcommand{\R}{\mathds{R}}

\newcommand{\wtilde}{\widetilde}
\def\centerarc[#1](#2)(#3:#4:#5){\draw[#1] ($(#2)+({#5*cos(#3)},{#5*sin(#3)})$) arc (#3:#4:#5);}

\newcommand{\dd}{\text{\rm d}}             

\def\II{\mathrm{I\kern-0.1emI}}
\def\III{\mathrm{I\kern-0.1emI\kern-0.1emI}}
\def\IV{\mathrm{I\kern-0.1emV}}

\DeclareMathOperator{\Supp}{supp}

\newcommand{\Aone}{$(\textrm{A1})$ }

\newcommand{\Athree}{$(\textrm{A3})$ }

\newcommand{\Asix}{$(\textrm{A6})$ }

\let\oldtocsection=\tocsection
\let\oldtocsubsection=\tocsubsection
\let\oldtocsubsubsection=\tocsubsubsection
\renewcommand{\tocsection}[2]{\hspace{0em}\oldtocsection{#1}{#2}}
\renewcommand{\tocsubsection}[2]{\hspace{1em}\oldtocsubsection{#1}{#2}}
\renewcommand{\tocsubsubsection}[2]{\hspace{2em}\oldtocsubsubsection{#1}{#2}}
\DeclareRobustCommand{\SkipTocEntry}[5]{}

\begin{document}

\title[]{A strong large deviation principle for the\\ empirical measure of  random walks}

\author{Dirk Erhard}
\address{\noindent UFBA\\
	Instituto de Matem\'atica, Campus de Ondina, Av. Adhemar de Barros, S/N. CEP 40170-110\\
	Salvador, Brazil}
\email{erharddirk@gmail.com}

\author{Tertuliano Franco}
\address{UFBA\\
	Instituto de Matem\'atica, Campus de Ondina, Av. Adhemar de Barros, S/N. CEP 40170-110\\
	Salvador, Brazil}
\curraddr{}
\email{tertu@ufba.br}
\thanks{}

\author{Joedson de Jesus Santana}
\address{UFBA\\
	Instituto de Matem\'atica, Campus de Ondina, Av. Adhemar de Barros, S/N. CEP 40170-110\\
	Salvador, Brazil}
\curraddr{}
\email{joedson.santana@ufba.br}
\thanks{}


\begin{abstract}
	In this article we show that the empirical measure of certain continuous time random walks satisfies a strong large deviation principle with respect to a topology introduced in~\cite{MV2016} by Mukherjee and Varadhan. This topology is natural in models which exhibit an invariance with respect to spatial translations. Our result applies in particular to the case of simple random walk and complements the results obtained in~\cite{MV2016} in which the large deviation principle has been established for the empirical measure of Brownian motion.
\end{abstract}

\maketitle

\tableofcontents

\section{Introduction}
Let $X=(X_t)_{t\geq 0}$ be a Markov process on a Polish space $(\Sigma, d)$. In the 1970's, Donsker and Varadhan~\cite[II]{DonskerVaradhanI} showed that under suitable conditions the empirical measure defined by
\begin{equation*}
L_t = \frac{1}{t} \int_0^t\one_{X_s} \dd s
\end{equation*}
satisfies a large deviation principle in the space of probability measures equipped with the weak topology given that $\Sigma$ is \emph{compact}. When the state space $\Sigma$ is \emph{not compact}, the upper bound holds for \emph{compact} sets rather than \emph{closed} sets. Meanwhile, the lower bound remains valid under a certain irreducibility assumption which we will detail further below.

In certain scenarios, this \emph{weak} large deviation principle can be upgraded to a standard (or \emph{strong}) one, thereby recovering the large deviation upper bound for all closed sets. However, such cases do not encompass many natural examples. To address the lack of compactness in applications, a confining drift may be added to the Markov chain (or diffusion)~\cite{DV83}, or it may be folded onto a large torus~\cite{B94,BH97,DV79}. Folding the Markov process on a large torus is often sufficient when one aims for more crude information like the behaviour of a partition function. Yet, if one is interested in finer properties of the model at hand the above mentioned methods are not enough to account for the lack of a strong large deviation principle.

 Recently, Mukherjee and Varadhan~\cite{MV2016} introduced a new approach by embedding the space of probability measures on $\mathbb{R}^d$ into a larger space equipped with a certain topology, rendering it a \emph{compact} metric space. Under this new topology, they established a \emph{strong} large deviation principle for the empirical measure of Brownian motion~\cite[Theorem 4.1]{MV2016}, which was then successfully applied to the \textit{Polaron Problem}~\cite{BKM2017, KM2017, MV2016}. The compactification of measures has also proven fruitful in the context of directed polymers, as demonstrated in~\cite{BatCha20,BroMuk19}. For an overview of Large Deviation Theory, readers may consult~\cite{DemboZei10:book,dH-book2000}, and for the role of topology in this theory, refer to~\cite{Var18}.

\subsection{Topology of Mukherjee and Varadhan}
In this section we will explain the topology of Mukherjee and Varadhan. In the original article~\cite{MV2016} the authors constructed a metric space $\wtilde{\mc X}$ as an enlargement of $\mc M_1(\R^d)$, the space of probability measures on $\R^d$. Since we are interested in a discrete context our construction starts with the space of probability measures $\mc M_1=\mc M_1(\Z^d)$ on $\Z^d$. The changes that need to be done to adapt from the original framework are minor and we outline those changes whenever necessary.

Denote by  $\wtilde{\mathcal{M}}_1  = \mathcal{M}_1 \ /\!\sim$ the quotient space of $\mathcal{M}_1$  under the action of $\Z^d$ (as an additive group on $\mathcal{M}_1$).
For any $\mu \in \mathcal{M}_1$, its orbit is defined by $\wtilde{\mu}=\{\mu*\delta_x : x\in \R^d\}\in\widetilde{\mathcal{M}}_1$.
For $k\geq 2$, we define $\mc F_k$ as the space of all functions $f:(\Z^d)^k\to \R$ that are diagonally translation invariant, i.e.,
\begin{equation*}
f(u_1 +x, \ldots, u_k+x)= f(u_1,\ldots, u_k)\,\qquad \forall\, x, u_1,\ldots, u_k\in\Z^d\,,
\end{equation*}
and vanishing  once getting away from the diagonal, i.e.,
\begin{equation*}
\lim_{\max_{i\neq j}|u_i-u_j|\to\infty} f(u_1,\ldots, u_k)=0\,.
\end{equation*}
For $k\geq 2$, $f\in\mc F_k$ and $\mu\in \mc M_{\leq 1}$, we define $\Lambda (f,\mu)$, the integral of $f$ with respect to a product measure of $k$ copies of $\mu$, by
\begin{equation*}
\Lambda (f,\mu):=\int f(u_1, \ldots, u_k)\prod_{1\leq i\leq k} \mu(\dd u_i)\,.
\end{equation*}
Due to the translation invariance of $f$, the above expression actually only depends on the orbit $\widetilde\mu$. The class of test functions that is of special importance is defined via
\begin{equation*}
\mc F=\bigcup_{k\geq 2}\mc F_k\,.
\end{equation*}
As noted in~\cite{MV2016}, there exists a countable dense set for each $\mc F_r$ (under the uniform metric). Taking the union of all of these subsets we obtain a countable set, which we will write as
\begin{equation*}
\big\{f_r(u_1, \ldots, u_{k_r}): r\in\bb N\big\}\,.
\end{equation*}
The desired compactification of $\wtilde{\mc M}_{\leq 1}$ is the space $\wtilde{\mc X}$ defined via
\begin{equation*}
\wtilde{\mc X}:=\Big\{\xi=\{\wtilde\alpha_i\}_{i\in I}\,:\, \wtilde\alpha_i\in\wtilde{\mc M}_{\leq 1}\,, \sum_{i\in I}\alpha_i(\Z^d)\leq 1\Big\}
\end{equation*}
equipped with the metric
\begin{equation*}
{\bf D}(\xi_1, \xi_2):=\sum_{r\geq 1}\frac{1}{2^r}\, \frac{1}{1+\|f_r\|_{\infty}}\,
\Big| \sum_{\wtilde\alpha\in\xi_1}\Lambda(f_r,\alpha)-\sum_{\wtilde\alpha\in\xi_2}
\Lambda(f_r,\alpha)\Big|\,.
\end{equation*}
Theorem 3.1 in~\cite{MV2016}  reads as follows.
\begin{theorem}\label{thm:metric}
	${\bf D}$ is a metric on $\wtilde{\mc X}$.
\end{theorem}
The second key result in ~\cite{MV2016} is its Theorem 3.2 which reads:
\begin{theorem}\label{thm:completion}
	The set of orbits $\wtilde{\mc M}_1$ is dense in $\wtilde{\mc X}$. Furthermore, given any sequence $(\wtilde{\mu}_n)_n$ in $\wtilde{\mc M}_1$, there is a subsequence that converges to a limit in $\wtilde{\mc X}$. Hence, $\wtilde{\mc X}$ is a compactification of $\wtilde{\mc M}_1$. It is also the completion under the metric ${\bf D}$ of $\wtilde{\mc M}_1$.
\end{theorem}
Theorem~\ref{thm:metric} follows from the same proof as in~\cite{MV2016}, Theorem~\ref{thm:completion} however requires some small modifications that we will outline in Section~\ref{S: 2.1}.
Before we close this section let us give an illustrative example.
Define a measure
\begin{equation*}
\mu_n= \frac12 {\mc U}(\{n-1,n,n+1\}) +\frac13 \mc U(\{-n-1,-n,-n+1\})+\frac16 \mc U(\{-n,-n+1,\ldots, n\})\,,
\end{equation*}
where $\mc U(A)$ denotes the uniform distribution on the set $A$.
Then, $(\mu_n)$ clearly does not converge in the weak topology. However, one can show that its orbit $(\wtilde\mu_n)$ converges in $(\wtilde{\mc X}, {\bf D})$ to
\begin{equation*}
\xi=\Big\{\frac12 \mc U_3(\cdot),\, \frac13 \mc U_3(\cdot)\Big\}\,.
\end{equation*}
Here $\mc U_3(\cdot)$ denotes the orbit of the uniform distribution on a set of three consecutive numbers.
We note that the contribution of $\tfrac16 \mc U(\{-n,-n+1,\ldots, n\})$ vanishes in the limit as $n\to\infty$.
\subsection{Main result}\label{sec:main}
Let $(X_t)_{t\geq 0}$ be a Markov chain on $\ds Z^d$ with generator $\mc L$ acting on  functions $f:\mathds{Z}^d\to \mathds{R}$ via
\begin{equation}\label{eq:gen}
(\mc Lf)(x)= \sum_{y\in \ds Z^d} a_{x,y}[f(y)-f(x)]\,.
\end{equation}
We denote the domain of $\mc L$ by $\mc D$ and we assume that the rates $(a_{xy})_{x,y\in \Z^d}$ satisfy the following six assumptions.
\begin{itemize}
	\item[$\rm{(A1)}$] The process $X=(X_t)_{t\geq 0}$ is a Feller process.
	\item[$\rm{(A2)}$] 	$\sup_{x\in \Z^d}\sum_{y\in \Z^d}a_{x,y} < \infty\,.$
	\item[$\rm{(A3)}$] For all $x,y\in \mathds{R}^d$ there exist $n\in \ds N$ and a sequence of points $x=x_0, x_1, \ldots, x_n= y$ such that for all $i\in\{0,1,\ldots, n-1\}$ we have that $a_{x_i,x_{i+1}}>0$.
	\item[$\rm{(A4)}$] For all $c,x,y\in\Z^d$ we have that $a_{x,y}= a_{x+c, y+ c}$.
	\item[$\rm{(A5)}$] There exists a positive function $(u_t)_{t\geq 0}$ satisfying $\lim_{t\to \infty} \tfrac1t \log u_t = 0$ such that
	\begin{equation*}
	\limsup_{t\to \infty} \frac1t \log \ds P\Big( \sup_{0\leq s\leq t} |X_s|\geq u_t\Big) = -\infty.
	\end{equation*}
	\item[$\rm{(A6)}$] The rates are symmetric in the sense that for all $x,y\in\Z^d$ one has that $a_{x,y}=a_{y,x}$.
\end{itemize}
\medskip
The third condition of course simply states that $(X_t)_{t\geq 0}$ is an irreducible Markov chain.
Next, we define the quantity of interest, namely, the \textit{normalized occupation measure}  $(L_t)_{t\geq 0}$ defined via
\begin{equation}
\label{eq:occupation_measure}
L_t(A)= \frac1t \int_0^t \one_A(X(s)) \, \dd s\,,\qquad A\subseteq \Z^d\,.
\end{equation}
Note that defined in this way, for each $t\geq 0$, the quantity $L_t$ is a random probability measure on $\Z^d$.
We recall that in~\cite{DonskerVaradhanIII}, under the assumptions $\rm{(A1)}$--$\rm{(A3)}$, it was shown that $(L_t)_{t\geq 0}$ satisfies a weak large deviation principle with respect to the usual weak convergence of probability measures with rate function $I$ defined on $\mc M_1(\Z^d)$ via
\begin{equation}
\label{eq:Iclassical}
I(\mu) = -\inf_{\substack{u> 0\\ u\in \mc D}} \int \Big( \frac{\mc L u}{u}\Big)(x) \mu(\dd x) =\big\Vert(-\mc L)^{\frac12} \sqrt{\mu}\big\Vert^2\
=\sum_{x,y\in\Z^d}a_{x,y}(\sqrt{\mu}(x)-\sqrt{\mu}(y))^2\,.
\end{equation}
where the norm appearing above is on the space $\ell^2(\Z^d)$. The second equality is thanks to Assumption~\Asix  and is a consequence of~\cite[I, Theorem 5]{DonskerVaradhanI}, while the last equality follows from a straightforward computation.
We note at this point that the above definition also makes sense for sub-probability measures $\mu\in \mc M_{\leq 1}(\Z^d)$ and we will make use of that without further mentioning it.
We remark that it might not be obvious to check from~\cite{DonskerVaradhanI} that Assumptions \Aone--\,\Athree are the correct assumptions to obtain the weak LDP. Therefore, in the forthcoming Subsection~\ref{subsec:weak} we will show that our assumptions are indeed the correct ones.

To proceed we extend the rate function $I$ to $\wtilde{\mc X}$. We define $\wtilde I:\wtilde{\mc X} \to \R$ for $\xi\in\wtilde{\mc X}$ via
\begin{equation}\label{eq:Inew}
\wtilde I(\xi)= \sum_{\wtilde\alpha\in\xi} I(\alpha)\,,
\end{equation}
where $\alpha$ is an arbitrary representative of $\wtilde\alpha$. This makes sense because we will show in Lemma~\ref{lem:PropertiesI} that, as a consequence of Assumption \rm{(A4)}, the rate function $I$ depends on $\mu$ only through its orbit. Finally denote by $\wtilde{L}_t$ the embedding of $L_t$ into $\wtilde{\mc X}$ and denote its law by $Q_t$.

Our main result then reads as follows.
\begin{theorem}
	\label{thm:main}
	Assume that Assumptions $\rm{(A1)}$--$\rm{(A6)}$ hold. Then, the family of measures $(Q_t)_{t\geq 0}$ defined on $\wtilde{\mc X}$ satisfy a large deviation principle with rate function $\wtilde I$ and rate $t$.
\end{theorem}

Before we come to the end of that section we will shortly discuss our assumptions. As mentioned above, Assumptions $\rm{(A1)}$--$\rm{(A3)}$ will guarantee that the family of empirical measures satisfies a weak large deviation principle in the usual weak topology in $\mc M_1$.  Note that
the topology of Mukherjee-Varadhan is applied to contexts in which the model at hand exhibits a shift invariance in space. In particular, the large deviation principle in $\wtilde{\mc X}$ can then be seen as an extension of the weak large deviation principle on $\mc M_1$, and one expects the new large deviation rate function to be related to the one on $\mc M_1$. It is therefore natural to consider models in which the rate function $I$ governing the weak large deviation principle actually only depends on the orbit of the measure at hand. Our Assumption~$\rm{(A4)}$ guarantees exactly that. Regarding Assumption~$\rm{(A5)}$ it should be satisfied in almost all examples and simply states that it is very costly for the random walk to travel exponentially large distances. This will come in handy for a coarse graining argument that will be used to establish the large deviation upper bound. Finally, we believe that Theorem~\ref{thm:main} should hold also without the Assumption~$\rm{(A6)}$. However, without Assumption~$\rm{(A6)}$ the second equality in~\eqref{eq:Iclassical} does not hold which makes the analysis  way more delicate.

\subsection{Structure of the paper}
In Section~\ref{sec:proof} we present the proofs of Theorems~\ref{thm:completion} and~\ref{thm:main}. To that end, we  first show in Section~\ref{subsec:weak} that our assumptions yield a weak large deviation principle in the usual weak topology in $\mc M_1$. In Section~\ref{S: 2.1} we  prove Theorem~\ref{thm:completion}. Then, in Section~\ref{subsec:weak} we  prove that our assumptions imply a weak large deviation principle in the usual weak topology in $\mc M_1(\Z^d)$. In the remaining subsection of Section~\ref{sec:proof} we  complete the proof of Theorem~\ref{thm:main}. Finally, in Section~\ref{sec:applications} we provide two  applications of our main result.

\section{Proof}\label{sec:proof}
In this section we prove Theorems~\ref{thm:main}. Recall that to establish a large deviation principle three properties need to be established. The lower semi-continuity of the rate function, the lower bound for open sets and the upper bound for closed sets. Our first task is to check that our hypothesis indeed imply the hypothesis of  Donsker and Varadhan \cite{DonskerVaradhanIII}, which in its hand imply a weak LDP, that is, a large deviations principle for compact sets in the usual topology of $\mc M_1(\Z^d)$.

\subsection{Adaptation of the proof of Theorem~\ref{thm:completion} and auxiliary results}\label{S: 2.1}

We start with the  proof of Theorem~\ref{thm:completion}, which is an adaptation of the proof of \cite[Theorem~3.2]{MV2016}. It is divided in  two steps:\medskip

\textit{Step 1.} $\wtilde{\mc M}_1$ is dense in $\wtilde{\mc X}$. \medskip

\textit{Step 2.} Any sequence $(\mu_n)_n$ has a subsequence that converges to some $\xi \in \wtilde{\mc X}$.\medskip

We start with Step 1. Let $\xi\in\wtilde{\mc X}$.
We are going to show that there exists a sequence of $(\wtilde{\mu}_n)$ of orbits in $\wtilde{\mc M}_1(\Z^d)$ such that $\wtilde{\mu}_n$ converges to $\xi$ in $\wtilde{\mc X}$.
Let $\xi=\{\wtilde{\alpha}_i:\, i\in I\}\in\wtilde{\mc X}$. As argued in~\cite[Theorem 3.2]{MV2016} we can assume that $\xi$ is a finite collection of orbits of sub-probability measures. For any $i\in I$ define $p_i=\alpha_i(\Z^d)$ and choose spatial points $(a_i)_{i\in I}=(a_i(n))_{i\in I}$ such that $\inf_{i\neq j}|a_i-a_j|\to\infty$ as $n\to\infty$. Also, let $\nu_n$ denote the uniform distribution of $[-n,n]^d\cap\Z^d$, i.e.,
\begin{equation}\label{eq:nun}
\begin{aligned}
\nu_n(x)=\frac{1}{(2n+1)^d}\begin{cases}
1, &\text{if }x\in [-n,n]^d\cap\Z^d\,,\\
0, &\text{otherwise.}
\end{cases}
\end{aligned}
\end{equation}
We then note that for any $f\in\mc F_k$,
\begin{equation*}
\lim_{n\to\infty}\int f(x_1,\ldots, x_k)\prod_{i=1}^k\nu_n(\dd x_i)=0\,.
\end{equation*}
Indeed, this is a consequence of the fact that $(\nu_n)$ totally disintegrates in the sense that, for any $r>0$,
\begin{equation*}
\lim_{n\to\infty}\sup_{x\in\Z^d}\nu_n(B(x,r))=0\,,
\end{equation*}
and Lemma 2.3 in~\cite{MV2016} whose proof carries over \textit{mutatis mutandis} to the present context.
We then define
\begin{equation*}
\mu_n=\sum_{i\in I} \alpha_i*\delta_{a_i} + \Big(1-\sum_{i\in I}p_i\Big)\nu_n\,.
\end{equation*}
Since $f\in\mc F_k$ vanishes whenever the distance of two coordinates tends to infinity and $\inf_{i\neq j}|a_i-a_j|\to\infty$  we see that
\begin{equation*}
\lim_{n\to\infty}\int f(x_1,\ldots, x_k)\prod_{i=1}^k\mu_n(\dd x_i)
=\sum_{i\in I}\int f(x_1,\ldots, x_k)\prod_{i=1}^k\alpha_i(\dd x_i)\,.
\end{equation*}
This shows that the sequence of orbits $(\wtilde{\mu}_n)$ converges to $\xi$ in $\wtilde{\mc X}$, and therefore $\wtilde{\mc M}_1$ is dense in $\wtilde{\mc X}$.

The second step of the proof in~\cite{MV2016} shows that every sequence $(\wtilde{\mu}_n)$ has a converging subsequence in $\wtilde{\mc X}$. The proof in the discrete context works  without any major adaptations, relying fundamentally on the following  lemma (\cite[Lemma 2.2]{MV2016}).
\begin{lemma}
	\label{lem:decomposition}
	Let $(\mu_n)_{n\in \bb N}$ be a sequence of sub-probability measures  that converges vaguely to a sub-probability measure $\alpha$, then  each $\mu_n$ can be written as $\mu_n=\alpha_n+\beta_n$. Here $(\alpha_n)_{n\in \bb N}$ converges in the weak topology to $\alpha$ and $(\beta_n)_{n\in \bb N}$ converges in the vague topology to zero. Moreover, $\alpha_n$ and $\beta_n$ can be chosen to have disjoint supports and such that the support of each $\alpha_n$ is compact.
\end{lemma}
The above formulation of Lemma~\ref{lem:decomposition} is a bit more general than in its original version in~\cite{MV2016} but its statement follows directly from the proof of~\cite[Lemma 2.2]{MV2016}. We also refer to~\cite{EP2023} where this was already noted.
Before we formulate the next result we will provide one more definition.
\begin{definition}
	Two sequences $(\alpha_n)_{n\in \bb N}$ and $(\beta_n)_{n\in \bb N}$ of sub-probability measures on $\Z^d$ are widely separated if for any function $W:\Z^d\to\R$ that vanishes at infinity
	\begin{equation*}
	\lim_{n\to\infty}\int W(x-y)\,\alpha_n(\dd x)\beta_n(\dd y) =0\,.
	\end{equation*}
\end{definition}
We note that, apart of the discrete scenario,  the above definition differs from the one given in~\cite[Subsection~2.4]{MV2016}, which requires $W$ to be strictly positive. However, as one can see in \cite[Lemma 2.4]{MV2016}, this condition of strictly positivity can be dropped, and  our definition is equivalent to theirs.
Lemma~\ref{lem:decomposition} together with the proof of the second step in~\cite[Theorem 3.2]{MV2016} then yields the following corollary.
\begin{corollary}\label{cor:decomposition}
	Let $(\wtilde{\mu}_{n})$ be a sequence in $\widetilde{\mc X}$, converging to $\xi =\{\wtilde{\alpha}_i\}_{i\in I}\in\widetilde{\mc X}$. Then there exists a sub-sequence which we also denote by $(\wtilde{\mu}_{n})$ such that for any $k\leq |I|$ we can write
	\begin{align*}
	\mu_n=\sum_{i=1}^{k}\alpha_{n,i}+\beta_n
	\end{align*}
	so that
	\begin{itemize}
		\item $(\alpha_{n,i})_{n\in \bb N},   i=1,\ldots, k,$ and $(\beta_n)_{n\in \bb N}$ are sequences of sub-probability measures in $\Z^d$;
		\item for each $i=1, \ldots, k$ there are sequences $(a_{n,i})_{n \in \bb N}$ of elements of $\Z^d$  such that
		\begin{align*}
		&\alpha_{n,i}\ast\delta_{a_{n,i}}\Rightarrow\alpha_i\in \wtilde{\alpha}_i,\quad  \text{ as } n\to \infty\,,\\
		&\lim_{n\to \infty}\min_{i\neq j} \vert a_{n,i}-a_{n,j}\vert=\infty\,,
		\end{align*}
		and $(\beta_n)_{n\in \bb N}$ is widely separated from each $(\alpha_{n,i})_{n\in \bb N}$;
		\item 	The supports of $\alpha_{n,1}, \alpha_{n,2}, \ldots, \beta_n$ are all disjoint and for each $i$ there exists a sequence $ (R_{n,i})_{n\in \bb N}$ tending to infinity such that
		$$\Supp (\alpha_{n,i})\subset B(-a_{n,i}, R_{n,i})$$
		and
		$$\Supp(\beta_{n,i})\subset \big[\cup_i B(-a_{n,i}, R_{n,i})]^\complement.$$
	\end{itemize}
\end{corollary}

We finish this section with one more result which is a consequence of the proof in~\cite[Theorem~3.2]{MV2016}.
\begin{lemma}
	\label{lem:convergenceinM1tilde}
	Assume that the sequence $(\wtilde{\alpha}_n)$ converges in $\wtilde{\mc M}_1$ to $\wtilde{\alpha}\in\wtilde{\mc M}_1$. Then $(\wtilde{\alpha}_n)$ converges in $\wtilde{\mc X}$ to $\wtilde{\alpha}$.
\end{lemma}
The proof of this result is a consequence of the analysis of the case $q=p$ in~\cite[Proof of Theorem~3.2]{MV2016} and it will be omitted here.
\subsection{Weak LDP}\label{subsec:weak}

We start by recalling the assumptions in~\cite{DonskerVaradhanIII}, keeping the original notation.
Denote by $p(t,x,\dd y)$ the transition probability of the Markov process $(X_t)_{t\geq 0}$ under consideration whose state space is a Polish space $X$. We will denote its law by $\ds P_x$ when started at $x$, and the corresponding expectation will be denoted by $\ds E_x$. The following four assumptions were put in place:
\begin{enumerate}
	\item[$\rm{(DV1)}$] The transition probability $p(t,x,\dd y)$ is Feller, having a density $p(t,x,y)$ with respect to a reference measure $\beta(\dd y)$.
	\item [$\rm{(DV2)}$] Denote by $(T_t)_{t\geq 0}$ the semi-group of $X$. Define
	\begin{align*}
	B_0=\{f\in C(X):\lim_{t\to 0}\|T_tf-f\|=0\},
	\end{align*}
	and
	\begin{align*}
	B_{00}=\{ f\in C(X):\lim_{t\to 0}\sup_x \ds E_x\big[|f(X_t)-f(x)|\big]=0  \}.
	\end{align*}
	Moreover, for $\alpha\in \mathcal{M}_1(X)$ such that $I(\alpha)<\infty$, $k\in \mathds{N}$ and fixed $f_1, \ldots, f_k\in B_{00}$ define a neighbourhood $\mathcal{N}_\alpha$ of $\alpha$ in $\mathcal{M}_1(X)$ via
	\begin{align*}
	\mathcal{N}_{\alpha}=\Biggl\{\mu \in \mathcal{M}_1(X): \bigg|\int_{X}f_j(x)[\mu(dx)-\alpha(dx)]\bigg|<\epsilon, 1\leq j\leq k\Biggr\}\,.
	\end{align*}
	It was then assumed that $p(t,x,\dd y)$ is such that every neighboorhod of $\alpha \in \mathcal{M}_1(X)$ contains a neighborhood of the form $\mc N_\alpha$.
	\item[$\rm{(DV3)}$] For any $\mu\in \mathcal{M}_1(X)$ and any bounded measurable function $f$, there exists a sequence $(f_n)_{n\in \bb N}\in B_0^{\bb N}$ such that $\|f_n\|\leq\|f\|$ for all $n\in\bb N$ and such that $f_n\to f$ almost everywhere with respect to $\mu$\,.
	\item[$\rm{(DV4)}$]  For every $x\in X, \sigma>0$ and $E\subset X$ such that $\beta(E)>0$, it holds that
	\begin{align*}
	\int_{0}^{\infty}\exp(-\sigma t)p(t,x,E)\dd t>0\,,
	\end{align*}
	where $p(t,x,E)$ is defined via
	\begin{equation*}
	p(t,x,E)= \ds P_x(X_t\in E)\,.
	\end{equation*}
\end{enumerate}

\medskip
\noindent
Below we argue why our assumptions $\rm{(A1)}, \rm{(A2)}$ and $\rm{(A3)}$ indeed imply $\rm{(DV1)}, \rm{(DV2)}, \rm{(DV3)}$ and $\rm{(DV4)}$:
\begin{description}
	\item[$\rm{(A1)}$ implies $\rm{(DV1)}$]
	The Feller property is clear by assumption and the reference measure $\beta$ is in our case simply the counting measure since our state space is discrete. Hence, $\rm{(DV1)}$ holds.
	\item[$\rm{(A2)}$ implies $\rm{(DV2)}$]
	We note that the usual topology of weak convergence of probability measures is induced by testing against continuous, bounded functions. Therefore, it remains to argue that $B_{00}$ coincides with $C_b(\Z^d)$ in our case. Fix $f\in C_b(\Z^d)$, assume that $X_0=x$ and define the first jump time via
	\begin{equation*}
	T_1=\inf\{t>0: X_t\neq x\}\,.
	\end{equation*}
	Since $T_1$ has an exponential distribution with parameter $\sum_{y}a_{xy}$ we can estimate
	\begin{eqnarray*}
		\mathds{E}_x\Big[|f(X_t)-f(x)|\Big]&=&\mathds{E}_x\Big[\vert f(X_t)-f(x)\vert \mathds{1}_{\{T_1>t\}}\Big]+\mathds{E}_x\Big[\vert f(X_t)-f(x)\vert\mathds{1}_{\{T_1\leq t\}}\Big] \\
		&\leq & 0+2\|f\|_\infty \mathds{P}_x(T_1\leq t)\\
		&=&2\|f\|_\infty\sum_ya_{xy}\int_{0}^{t}\exp\Big\{-s\sum_ya_{xy}\Big\}ds\,.
	\end{eqnarray*}
	By $\rm{(A2)}$ the latter goes to zero uniformly in $x$ as $t\to 0$. Thus, $f\in B_{00}$ and $\rm{(DV2)}$ holds.
	\item[$\rm{(A2)}$ implies $\rm{(DV3)}$]
	Since the state space $\Z^d$ is a discrete space, we have that every function $f:\Z^d \to \R$ is continuous, and if this function is bounded then by the above arguments it clearly belongs to $B_0$. Thus, it is sufficient to take $f_n=f$, and $\rm{(DV3)}$ follows.
	\item[$\rm{(A3)}$ implies $\rm{(DV4)}$]
	This is a direct consequence of the fact that $\rm{(A3)}$ simply states that $(X_t)_{t\geq 0}$ is irreducible.
\end{description}

\subsection{Lower bound}
In this section we give the proof of the large deviation lower bound. More precisely, we will show the following result.
\begin{proposition}[Lower bound]\label{lem:lowerbound}
	For any open set $G\in \widetilde{\mc X}$
	\begin{align}\label{LB}
	\liminf_{t\to\infty}\dfrac{1}{t}\log Q_t(G)\geq-\inf_{\xi\in G}\widetilde{I}(\xi).
	\end{align}
\end{proposition}
For the proof we will need some properties of the rate function that are provided in the next lemma.
\begin{lemma}\label{lem:PropertiesI}
	The rate function $I$ defined in \eqref{eq:Iclassical} on the space $\mc M_{\leq 1}(\Z^d)$ is translation invariant, i.e., for any $c\in \Z^d$ one has that $I(\mu)= I(\mu*\delta_c)$, and it is homogeneous of degree 1, i.e., for any $\lambda\geq 0$ one has that $I(\lambda\mu)=\lambda I(\mu)$, hence it is convex and consequently sub-additive.
\end{lemma}
\begin{lemma}
	\label{lem:approx}
	Let $\xi \in \widetilde{\mc X}$ with $\widetilde{I}(\xi)< \infty$. Then, there exists a sequence $(\xi_n)$ in $\widetilde{\mc M}_1(\Z^d)$ such that
	\begin{equation}\label{LBC}
	\limsup_{n\to \infty}\wtilde{I}(\xi_n)\leq \wtilde{I}(\xi).
	\end{equation}
\end{lemma}

The proofs of the above lemmas are deferred to the end of that section. We  now provide the proof of Proposition~\ref{lem:lowerbound}.
\begin{proof}[Proof of Proposition~\ref{lem:lowerbound}]
	Note that to show \eqref{LB} it is enough to prove that, given $\xi \in \widetilde{\mc X}$ with $\widetilde{I}(\xi)<\infty$,
	\begin{align}\label{LB2}
	\liminf_{t\to\infty}\dfrac{1}{t}\log Q_t(U)\geq -\widetilde{I}(\xi)
	\end{align}
	for any \textcolor{blue}{open} neighbourhood $U$ of $\xi$.
	First of all note that for any open neighbourhood $U$ of $\xi$ the set $U \cap \widetilde{\mathcal{M}}_1$ is open in $\widetilde{\mc M}_1$. Indeed, given any $\theta\in U\cap \widetilde{\mathcal{M}}_1$ and any sequence $(\xi_n)_{n\in\bb N}$ in $\widetilde{\mc M}_1$ that converges to $\theta$ with respect to the topology in $\widetilde{\mc M}_1$ it follows from Lemma~\ref{lem:convergenceinM1tilde} that $\xi_n$ converges also with respect to the topology of $\widetilde{\mc X}$ to $\theta$. Thus, since $U$ is open in $\wtilde{\mc X}$ eventually $\xi_n\in U$, and since it is a sequence in $\widetilde{\mc M}_1$ we have therefore that $\xi_n\in U\cap\widetilde{\mc M}_1$ provided that $n$ is sufficiently large. Thus, $U\cap\widetilde{\mc M}_1$ is open in $\widetilde{\mc M}_1$ as claimed. Since $\widetilde{L}_t\in\widetilde{\mathcal{M}}_1$ we have that
	$$
	\mathds{P}(\widetilde{L}_t\in U)=\mathds{P}(\widetilde{L}_t\in U\cap \widetilde{\mathcal{M}}_1).
	$$
	Define
	\begin{align*}
	q:\;&\mathcal{M}_1\to\widetilde{\mathcal{M}}_1\\\
	&u\mapsto q(u)=\wtilde{u}\,,
	\end{align*}
	as the canonical map from $\mc M_1$ to $\wtilde{\mc M}_1$.
	Note that $q$ is a continuous mapping, so that $q^{-1}(U \cap\widetilde{\mc M}_1)$ is an open set.
	Now consider a sequence $(\xi_n)_{n\in \bb N}$ as in the claim above. Then for all $n$ sufficiently large we have that $\xi_n\in U\cap \widetilde{\mc M}_1$. For those $n$ we can estimate
	\begin{align*}
	\liminf_{t\to\infty}\frac{1}{t}\log\mathds{P}(\widetilde{L}_t\in U)&=
	\liminf_{t\to\infty}\frac{1}{t}\log\mathds{P}(L_t\in q^{-1}(U\cap\widetilde{\mathcal{M}}_1))\\
	&\geq -I(\xi_n)\\
	&=-\widetilde{I}(\xi_{n}).
	\end{align*}
	Here, we used in the inequality the fact that $(L_t)_{t\geq 0}$ satisfies the large deviation lower bound in the usual weak topology of probability measures together with the fact that by Lemma~\ref{lem:PropertiesI} the rate function $I$ only depends on the orbit of the measure, and the last equality uses that $\xi_n$ as an element of $\wtilde{\mc X}$ contains only one element. To conclude it suffices to send $n$ to infinity.
\end{proof}

We now come to the proof of Lemma~\ref{lem:PropertiesI}.
\begin{proof}[Proof of Lemma~\ref{lem:PropertiesI}]
	We first show that $I$ is translation invariant.
	We write, using Assumption~$\rm{(A4)}$ in the last equality
	\begin{equation*}
	\begin{aligned}
	\sum_{x,y\in\Z^d}a_{x,y}(\sqrt{\mu}(x+c)-\sqrt{\mu}(y+c))^2
	&= \sum_{x,y\in\Z^d}a_{x-c,y-c}(\sqrt{\mu}(x)-\sqrt{\mu}(y))^2\\
	&=\sum_{x,y\in\Z^d}a_{x,y}(\sqrt{\mu}(x)-\sqrt{\mu}(y))^2\,.
	\end{aligned}
	\end{equation*}
	Homogeneity, and convexity are direct consequences of the definition of $I$. The sub-additivity is then a immediate consequence of the homogeneity and the convexity. Hence, we can conclude.
\end{proof}

We now come to the proof of Lemma~\ref{lem:approx}.
\begin{proof}[Proof of Lemma~\ref{lem:approx}]
	By our analysis at the beginning of Section~\ref{S: 2.1} the sequence $\wtilde{\mu}_n$ approximating $\xi$ in $\wtilde{\mc X}$ can be chosen to be of the form
	$$
	\mu_{n}=\sum_{i \in J_n}
	\alpha_i*\delta_{a_i}+\bigg(1-\sum_{i\in J_n}p_i \bigg)\nu_n\,,
	$$
	where $\nu_n$ is defined in~\eqref{eq:nun}.
	Here, given $\xi=\{\alpha_i:\, i\in I\}$ as at the beginning of the proof we denote by $(J_n)$ a suitable chosen sequence of subsets of $I$ that tend to $I$ in the limit as $n$ tends to infinity (note that in Section~\ref{S: 2.1} the first step we detailed was for the case in which $I$ was finite).
	By Lemma~\ref{lem:PropertiesI}  the rate function $I$ is sub-additive on $\mathcal{M}_{\leq 1}(\Z^d)$, thus
	\begin{align*}
	I(\mu_{n})&=I\Bigg(\sum_{i\in J_n}(\alpha_i*\delta_{a_i})) +\bigg(1-\sum_{i\in J_n}p_i\bigg)\nu_n\Bigg)  \\
	&\leq\sum_{i\in J_n}I(\alpha_i*\delta_{a_i})+\bigg(1-\sum_{i\in J_n}p_i\bigg)I(\nu_n)\\
	&=\sum_{i\in J_n}I(\alpha_i)+\bigg(1-\sum_{i\in J_n}p_i\bigg)I(\nu_n)\,.
	\end{align*}
	A straightforward computation shows that $\lim_{n\to\infty}I(\nu_n)=0$.
	Since, by the definition of $\wtilde{I}$ in \eqref{eq:Inew}
	\begin{equation*}
	\sum_{i\in J_n}I(\alpha_i)\leq \wtilde{I}(\xi)\,,
	\end{equation*}
	we can conclude.
\end{proof}

\subsection{Upper bound}
\begin{proposition}[Upper bound]\label{prop:UB} For any closed set $F$ in $\widetilde{\mc X}$, we have that
	\begin{align}\label{eq:Upp}
	\limsup_{t\to \infty}\dfrac{1}{t}\log Q_t(F) \leq -\inf_{\xi\in F}\widetilde{I}(\xi).
	\end{align}
\end{proposition}

Before diving into the proof we present some definitions and observations that will be useful for our purposes.
Let $\mathcal{U}$ be the space of functions of the form $u=c+v$ where $v$ is a smooth non-negative function with compact support on $\R^d$ and $c>0$ is a constant. Let $\varphi$ be a function satisfying $0\leq\varphi(x)\leq1, \ \varphi(x)=1$ inside the unit ball and $\varphi(x)=0$ outside the ball of radius 2. For any $k\geq 1, R>0, u_1, ... ,u_k\in \mathcal{U} $ and $a_1,  ..., a_k\in \Z^d$ and $c>0$ we consider the function
\begin{align}\label{eq:g}
g(x)=g(k, R, c, a_1, ..., a_k, x)=c+\sum_{i=1}^{k}	u_i(x+a_i)\varphi\bigg(\dfrac{x+a_i}{R}\bigg)
\end{align}
and define $F:\Omega\to\R$ by setting
\begin{align}\label{eq:F}
\nonumber	F(u_1, \ldots, u_k, c, R, t, \omega)&=
\sup_{\substack{a_1, \ldots, a_k\,:\\ \inf_{i\neq j}|a_i-a_j|\geq 4R}} \
\dfrac{1}{t}\int_{0}^{t}\dfrac{\mathcal{-L} (g(X(s)))}{g(X(s))}\dd s\\
\
&=		\sup_{\substack{a_1, \ldots, a_k\,:\\ \inf_{i\neq j}|a_i-a_j|\geq 4R}} \
\int_{\R^d}\dfrac{\mathcal{-L} ( g(x))}{g(x)}L_t(\dd x).
\end{align}
Since the last expression depends only on the image $\widetilde{L}_t$ of $L_t$ in $\widetilde{\mc X}$, we can write
\begin{align*}
\widetilde{F}(u_1, u_2, \ldots, u_k, c, R, \widetilde{L}_t)
:= F(u_1, u_2,\ldots, u_k, c, R,t,\omega).
\end{align*}
We will prove first that $\widetilde{F}(\cdot)$ grows only sub-exponentially as $t\to \infty$.

\begin{lemma}\label{lem:UBfundamental}
	For any $k\geq1, R>0, \ u_1, \ldots, u_k\in \mathcal{U}$ and $c>0$,
	\begin{equation}\label{FSexp}
	\limsup_{t\to \infty}\dfrac{1}{t}\log \mathds{E}\{\exp\{t\widetilde{F}(u_1,\ldots, u_k, c, R, \widetilde{L}_t)\}\}\leq 0.
	\end{equation}
\end{lemma}
\begin{proof}
	The proof proceeds in two steps. In the first step we show that the result follows if there were no supremum over $a_1, a_2,\ldots, a_k$ in the definition of $\wtilde F$. To be more precise we will actually show that
	\begin{align*}
	\mathds{E}\biggl\{\exp \bigg\{\int_{0}^{t}\dfrac{(\mathcal{-L}g)(X_s)}{g(X_s)}\dd s\bigg\}\biggr\}
	\end{align*}
	is bounded from above uniformly in the choice of $a_1, \ldots, a_k\in \Z^d$.
	By the Feynman-Kac formula, the function
	\begin{align*}
	\Psi(t,x)=\mathds{E}_x\bigg\{g(X_t)\exp\bigg\{\int_{0}^{t}\dfrac{(\mathcal{-L} g)(X_s)}{g(X_s)}\dd s\bigg\}\bigg\}
	\end{align*}
	is a solution to the initial value problem
	\begin{align*}
	\left\{
	\begin{array}{l}
	\dfrac{\partial}{\partial t}\Psi(t,x)=(\mathcal{L}\Psi)(t,x) -\dfrac{(\mathcal{L} g)(x)}{g(x)}\Psi(t,x)\,,\\
	\Psi(0,x)=g(x).
	\end{array}
	\right.
	\end{align*}
	Moreover, it is known that the above problem has a unique solution.
	It is immediate that $\Psi(t,x)=g(x)$ solves the above heat equation. Furthermore, by definition \eqref{eq:g} we have that $g(x)\geq c.$
	Therefore, by the uniqueness of the solution
	\begin{equation*}
	g(x)=\mathds{E}_x\bigg\{g(X_t)\exp\bigg\{\int_{0}^{t}\dfrac{(\mathcal{-L}  g)(X_s)}{g(X_s)}\dd s\bigg\}\bigg\}
	\geq c\mathds{E}_x\bigg\{ \exp\bigg\{ \int_{0}^{t}\dfrac{(\mathcal{-L}  g)(X_s)}{g(X_s)}\dd s\bigg\}\bigg\}\,,
	\end{equation*}
	and consequently
	\begin{equation}\label{42}
	\mathds{E}_x\bigg\{ \exp \bigg\{\int_{0}^{t}\dfrac{(\mathcal{-L} g)(X_s)}{g(X_s)}\dd s\bigg\}\bigg\}\leq \dfrac{g(x)}{c}.
	\end{equation}
	Since $g$ is uniformly bounded from above in $a_1, \ldots, a_k\in \Z^d$ the claim follows.
	We come to the second step of the proof, namely we will deal with the supremum over $(a_1,...,a_k)$ inside the expectation. To that end we will use a coarse graining argument.
	First note that if the range of the random walk in the time interval $[0,t]$ is $r_t$, once any $|a_i| $ exceeds $r_t+2R$ it will no longer affect the value of $g$, since in this case we have that
$
	|\tfrac{X_s+a_i}{R}|\geq 2
$
	and $\varphi$  is supported in the ball of radius 2.
	Thus, we can limit each $a_i$ to the ball of radius $r_t+2R.$
	By Assumption \rm{(A5)} there exists a positive function $(u_t)_{t\geq 0}$ with $\lim_{t\to\infty}\tfrac1t \log u_t=0$ and a function $M(t)$ tending to infinity as $t\to\infty$ such that
	\begin{align*}
	\mathds{P}\big(\sup_{0<s\leq t}|X_s|\geq u_t\big)
	\leq\exp(-tM(t))
	\end{align*}
	for all $t\geq 0$.
	We will see that this will allow us to ignore the contributions of those trajectories satisfying $\sup_{0<s\leq t}|X_s|\geq u_t$. We can assume that $2R\leq u_t$, so that we can restrict the $a_i$'s to balls of radius $2u_t$.
	Moreover, the function
	\begin{align*}
	\dfrac{(\mathcal{-L} g)(x)}{g(x)}
	\end{align*}
	is bounded from above by some constant $c_2$. This is a consequence Assumption~\rm{(A2)} and the fact that $g$ is bounded from above and is bounded from below by $c$.
	With the above considerations in mind we can now estimate
	\begin{align*}
	&\mathds{E}\Bigg\{\exp\Bigg\{\underset{|a_i-a_j|\geq4R \ \forall i\neq j}{{\sup_{a_1, \ldots, a_k\in\Z^d}}}\int_{0}^{t}\dfrac{(\mathcal{-L}  g)(X_s)}{g(X_s)}\dd s \Bigg\}\Bigg\}\\
	\
	&\leq\mathds{E}\Bigg\{\exp\Bigg\{\underset{|a_i-a_j|\geq4R \ \forall i\neq j}{{\sup_{|a_1|\leq 2u_t, \ldots, |a_k|\leq 2u_t}}}\int_{0}^{t}\dfrac{(\mathcal{-L} g)(X_s)}{g(X_s)}\dd s \Bigg\}\Bigg\}+\exp ({c_2t})\mathds{P}\Big(\sup_{0\leq s\leq t}|X_s| \geq u_t\Big)\\
	&\leq \mathds{E}\Bigg\{\sum_{\substack{(a_1, \ldots, a_k)\in \Z^d\,\\|a_i-a_j|\geq 4R }}\exp\Bigg\{\int_{0}^{t}\dfrac{(\mathcal{-L} g)(X_s)}{g(X_s)}\dd s \Bigg\}\Bigg\} +\exp(c_2 t-tM(t))\\
	\ \ \ \
	&\leq (2u_t)^{dk}\sup_{\substack{(a_1,\ldots, a_k)\in \Z^d\,\\ |a_i-a_j|\geq 4R}}\mathds{E}\Bigg\{\exp\Bigg\{\int_{0}^{t}\dfrac{\mathcal{-L} ( g(X_s))}{g(X_s)}\dd s \Bigg\}\Bigg\} +\exp(c_2 t-tM(t))\,.
	\end{align*}
	Note that by the first step, the expectation above can be bounded by $g(k,R,c,a_1,\ldots, a_k,x)/c$ which actually is bounded uniformly in the choice of $a_1,\ldots, a_k$. Hence, taking the logarithm above, dividing by $t$, sending $t$ to infinity  shows that the left hand side in~\eqref{FSexp} is at most zero.

\end{proof}

\begin{lemma}\label{lem:Ftildelower}
	Let $(\wtilde{\mu}_{n})$ be a sequence in $\widetilde{\mc X}$ which converges to $\xi=\{\wtilde{\alpha}_j\}\in \wtilde{\mc{X}}$. For any $k\in\mathds{N}, i=1, \ldots,k$ and $u_{i,R}(x)=u_i(x)\varphi \bigg(\dfrac{x}{R}\bigg)$, where $u_i\in \mathcal{U}$, we have
	\begin{align} \label{eq:Lambda}
	\nonumber	\liminf_{n\to\infty}\widetilde{F}(u_1, \ldots,u_k, c, R, \wtilde{\mu}_n)&\geq\sup_{\alpha_1, \ldots, \alpha_k\in \xi}\sum_{i=1}^{k}\sup_{b\in\Z^d}\int\dfrac{(\mathcal{-L}u_{i,R})(x)}{c+u_{i,R}(x)}\alpha_i(dx+b) \\
	&=:\wtilde{\Lambda}(\xi, R, c, u_1, \ldots, u_k).
	\end{align}
\end{lemma}
Before diving into the proof of the above lemma we need to formulate a technical result, which will be useful in the proof of Lemma~\ref{lem:Ftildelower} and for the forthcoming results.
\begin{lemma}\label{lem:technical}
	Let $u$ be a non-negative function with compact support and $(\alpha_n)_{n\in \mathds{N}}$ be a sequence of sub-probability measures such that
	\begin{equation}\label{eq:supp}
	\lim_{n\to\infty}\mathrm{dist}(\Supp(\alpha_n),\Supp(u))=\infty\,,
	\end{equation}
	where $\mathrm{dist}(A,B)$ denotes the distance in the $\ell^1$-norm between two sets $A,B\subseteq\Z^d$.
	Under the above condition the following holds for all $c>0$,
	\begin{equation}\label{eq:Lutozero}
	\lim_{n\to\infty}\int \frac{-(\mc L u)(x)}{c+u(x)}\alpha_n(\dd x)=0\,.
	\end{equation}
\end{lemma}
We first prove Lemma~\ref{lem:Ftildelower} accepting the validity of Lemma~\ref{lem:technical}. The proof of the latter will be given afterwards.
\begin{proof}
	Fix $b\in \R^d$ and $\alpha_{1},\alpha_{2}, \ldots, \alpha_{k}\in \xi$. We will follow the second step in the proof of~\cite[Theorem~3.2]{MV2016}. By Corollary~\ref{cor:decomposition} we can find $\ell\geq k$ and a subsequence $\mu_{n_k}$ that we will suppress from the notation such that we have the decomposition
	$$\mu_{n_k}=\sum_{j=1}^{\ell}\alpha_{n,j}+\beta_n$$
	so that for a suitable choice of $a_{n_{i,j}}$ that satisfy
	$$
	\lim_{n\to \infty}\vert a_{n,i}-a_{n,j}\vert=\infty \ \  \text{for} \  i\neq j
	$$ one additionally has that
	\begin{equation}\label{eq:b}
	\begin{aligned}
	\alpha_{n,j}* \delta_{a_{n,j}+b}\to \alpha_{j}* \delta_b\,,\qquad \text{for }\qquad j\in \{1, 2, \ldots, k\}\,,
	\end{aligned}
	\end{equation}
	and such that the supports of the $\alpha_{n,j}$'s and $\beta_n$'s are all disjoint.

	\ \ \ \ \ \ \ \ \ \ \

	For $n$  large enough we have that $|a_{n,i}-a_{n,j}|\geq4R $ so that the supports of the $\{u_{i, R}(\cdot-a_{n,i})\}_i$ are mutually disjoint.
	Now defining
	\begin{align*}
	\Supp (u_{i,R}(\cdot-a_{n,i}))=U_i, \ \forall i=1, \ldots, k
	\end{align*}
	and fixing $x\in\Z^d$
	we have the two following cases:
	\begin{itemize}
		\item There exist $j\in \{1, \ldots, k\}$ such that $x \in U_j$ and consequently $x\notin U_i\ \ \forall i\neq j$. In this case we have that
		\begin{align*}
		\dfrac{\mathcal{-L}g(k,R,c, a_1, \ldots a_k,x)}{g(k, R,c,a_1, \ldots, a_k,x)}
		&=\dfrac{\sum_{i=1}^{k}\mathcal{-L}u_{i,R}(x-a_{n,i})}{c+\sum_{i=1}^{k}u_{i,R}(x-a_{n,i})}\\
		&=\dfrac{\sum_{i=1}^{k}\mathcal{-L}u_{i,R}(x-a_{n,i})}{c+u_{j,R}(x-a_{n,j})}\\
		&=\sum_{i=1}^{k}\dfrac{\mathcal{-L}u_{i,R}(x-a_{n,i})}{c+u_{j,R}(x-a_{n,j})}.
		\end{align*}
		\item $\forall j \in \{ 1, \ldots, k\}$ we have that $x \notin U_j$. Then, $u_{j,R}(x-a_i)=0 $ for all $j$ and in this case we have that
		\begin{align*}
		\dfrac{\mathcal{-L}g(k,R,c, a_1, \ldots a_k,x)}{g(k, R,c,a_1, \ldots, a_k,x)}
		&=\dfrac{\sum_{i=1}^{k}\mathcal{-L}u_{i,R}(x-a_{n,i})}{c+\sum_{i=1}^{k}u_{i,R}(x-a_{n,i})}\\
		&=\dfrac{\sum_{i=1}^{k}\mathcal{-L}u_{i,R}(x-a_{n,i})}{c}\\
		&=\sum_{i=1}^{k}\dfrac{\mathcal{-L}u_{i,R}(x-a_{n,i})}{c+u_{i,R}(x-a_{n,i})}.
		\end{align*}
	\end{itemize}
	Summing up the two contributions plus some elementary calculations yield that
	\begin{align*}
	\dfrac{\mathcal{-L}g(k,R,c, a_1, \ldots a_k,x)}{g(k, R,c,a_1, \ldots, a_k,x)}
	&=
\sum_{j=1}^{k}\sum_{i=1}^{k}\dfrac{\mathcal{-L}u_{i,R}(x-a_{n,i})}{c+u_{j,R}(x-a_{n,j})}\mathds{1}_{\{x\in U_j\}}
	+ \sum_{i=1}^{k}\dfrac{\mathcal{-L}u_{i,R}(x-a_{n,i})}{c+u_{i,R}(x-a_{n,i})}\mathds{1}_{\{x\notin \cup_{j=1}^{k}U_j\}}\\
	&= \sum_{j=1}^{k}\mathds{1}_{\{x\in U_j\}}\Biggl\{	\sum_{i=1,i\neq j}^{k}\dfrac{\mathcal{-L}u_{i,R}(x-a_{n,i})}{c+u_{j,R}(x-a_{n,j})} + \dfrac{\mathcal{-L}u_{j,R}(x-a_{n,j})}{c+u_{j,R}(x-a_{n,j})}\Biggr\}\\
	&
	+\mathds{1}_{\{x\notin \cup_{j=1}^{k}U_j\}} \sum_{i=1}^{k}\dfrac{\mathcal{-L}u_{i,R}(x-a_{n,i})}{c+u_{i,R}(x-a_{n,i})}\\
	&=\sum_{j=1}^{k}\mathds{1}_{\{x\in U_j\}}	\sum_{i=1,i\neq j}^{k}\dfrac{\mathcal{-L}u_{i,R}(x-a_{n,i})}{c+u_{j,R}(x-a_{n,j})}
	+\sum_{j=1}^{k} \dfrac{\mathcal{-L}u_{j,R}(x-a_{n,j})}{c+u_{j,R}(x-a_{n,j})}\\
	&	+\sum_{j=1}^{k}\sum_{i=1, i\neq j}^{k}
	\mathds{1}_{\{x\in U_j\}} \dfrac{\mathcal{L}u_{i,R}(x-a_{n,i})}{c+u_{i,R}(x-a_{n,i})}.
	\end{align*}
	For $x\in U_j$ we have that
	\begin{equation}\begin{split}\label{eq:firsttozero}
	\int_{ U_j}\bigg|\dfrac{\mathcal{-L}u_{i,R}(x-a_{n,i})}{c+u_{j,R}(x-a_{n,j})}\bigg|\mu_{n}( dx)
	&=\int_{U_j}\bigg|\dfrac{\sum_{y\in U_i}a_{x,y}[u_{i,R}(x-a_{n,i})-u_{i,R}(y-a_{n,i})]}{c+u_{j,R}(x-a_{n,j})}\bigg|\mu_{n}(dx)\\
	&=\int_{U_j}\dfrac{\sum_{y\in U_i}a_{x,y}u_{i,R}(y-a_{n,i})}{c+u_{j,R}(x-a_{n,j})}\mu_{n}(dx).
	\end{split}
	\end{equation}
	Using that $u$ is non-negative and bounded from above and that each $\mu_n$ is a sub-probability measure we can estimate the above by some proportionality constant times
\begin{equation}
    \sum_{y:\;|y-x|\geq d(U_i,U_j)}a_{x,y}\,.
\end{equation}
	The latter goes to zero as $n\to\infty$ by  $\rm{(A2)}$.
	By the  same argument, we have that
	\begin{align*}
	\int_{ U_j}\sum_{\substack{i=1\\ i\neq j}}^{k}\dfrac{\mathcal{-L}u_{i,R}(x-a_{n,i})}{c+u_{i,R}(x-a_{n,i})}\mu_{n}(dx)
	\end{align*}
	tends to zero as $n\to\infty$. Thus,
	\begin{equation}
	\begin{split}
	\liminf_{n\to\infty}\int \dfrac{\mathcal{-L}g(x)}{g(x)}\mu_{n}(dx)=&\liminf_{n\to\infty}\bigg(\int \sum_{j=1}^{k}\mathds{1}_{\{x\in U_j\}}	\sum_{i=1,i\neq j}^{k}\dfrac{\mathcal{-L}u_{i,R}(x-a_{n,i})}{c+u_{j,R}(x-a_{n,j})}\mu_{n}(dx)\\
	&+\int\sum_{j=1}^{k} \dfrac{\mathcal{-L}u_{j,R}(x-a_{n,j})}{c+u_{j,R}(x-a_{n,j})}\mu_{n}(dx)\\
	&+\int\sum_{j=1}^{k}\sum_{\underset{i\neq j}{i=1}}^{k}
	\mathds{1}_{\{x\in U_j\}} \dfrac{\mathcal{-L}u_{i,R}(x-a_{n,i})}{c+u_{i,R}(x-a_{n,i})}\,\mu_n(dx)\,\Big)\,.
	\end{split}\end{equation}
	This can be lower bounded by
	\begin{align}
&\liminf_{n\to\infty}\int \sum_{j=1}^{k}\mathds{1}_{\{x\in U_j
		\}}	\sum_{i=1,i\neq j}^{k}\dfrac{\mathcal{-L}u_{i,R}(x-a_{n,i})}{c+u_{j,R}(x-a_{n,j})}\mu_{n}(dx)\label{eq:2.17}\\
	&+\liminf_{n\to\infty}\int\sum_{j=1}^{k}\dfrac{(\mathcal{-L}u_{j,R})(x-a_{n,j})}{c+u_{j,R}(x-a_{n,j})}\mu_{n}(dx)\label{eq:2.18}\\
	&+\sum_{j=1}^{k}\liminf_{n\to\infty} \int\sum_{\underset{i\neq j}{i=1}}^{k}
	\mathds{1}_{\{x\in U_j\}} \dfrac{\mathcal{-L}u_{i,R}(x-a_{n,i})}{c+u_{i,R}(x-a_{n,i})}\mu_n(dx)\label{eq:2.19}	\\
	&=\sum_{j=1}^{k}\int\dfrac{\mathcal{-L}u_{j,R}(x)}{c+u_{j,R}(x)}\alpha_j(dx)\,.
	\end{align}
	Here, we used the considerations after Equation~\eqref{eq:firsttozero} to show that \eqref{eq:2.17} and \eqref{eq:2.19} converge to zero, and we used Lemma~\ref{lem:technical} to deal with \eqref{eq:2.18}.
	Since this argument works for any collection $\alpha_1, \ldots, \alpha_k\in\xi$ and since we can redefine $a_{n,i}':=a_{n,i}+b_i$, using~\eqref{eq:b} we can conclude that
	\begin{align*}
	\liminf_{n\to\infty}\widetilde{F}(u_1, \ldots,u_k, c, R, \wtilde{\mu}_n)&\geq\sup_{\alpha_1, ..., \alpha_k\in \xi}\sum_{j=1}^{k}\sup_{b\in\R^d}\int\dfrac{\mathcal{-L}u_{j,R}(x)}{c+u_{j,R}(x)}\alpha_j(d x+b)
	\end{align*}
	and the lemma is proved.
\end{proof}

Proof of Lemma \ref{lem:technical}.  Let $u \in \mathcal{U}$ and let $(\alpha_{n})$ be a sequence of sub-probability measures.
We have that
\begin{align*}
(\mathcal{L}u)(x)= \sum_{y}a_{x,y}[u(y)-u(x)].
\end{align*}
By \eqref{eq:supp} we have that $ \Supp(\alpha_{n})\cap\Supp(u)=\varnothing$
for  $n$ large enough. For such $n$ we can write
\begin{align*}
\int\dfrac{-(\mathcal{L}u)(x)}{c+u(x)}\alpha_{n}(d x)=&\int \dfrac{-(\mathcal{L}u)(x)}{c}\alpha_{n}(dx)\\
&	\leq \dfrac{1}{c}\sup_{x\in \Z^d} \sum_{ y\in\Supp(\alpha_n)} a_{x,y} \,.
\end{align*}
The latter term goes to zero by Assumption~\rm{(A2)} and the hypothesis of the lemma.


\begin{lemma}\label{lem:ratefct}
	For $\wtilde{\Lambda}$ defined in \eqref{eq:Lambda} and $\wtilde{I}$ defined in \eqref{eq:Inew}, we can write
	\begin{align*}
	\wtilde{I}(\xi)=	\sup_{\substack{R, c>0, \ k\in \mathds{N}\\ u_1, \ldots, u_k\in U}}\wtilde{\Lambda}(\xi,R,c, u_1, \ldots, u_k)
	\end{align*}
\end{lemma}
\begin{proof}
	By the classical rate function $I$ defined in  \eqref{eq:Iclassical} for any $\alpha\in \mathcal{M}_{\leq 1}(\R^d)$ we can identify $I$ as
	\begin{align*}
	I(\alpha) = \sup_{\substack{c> 0\\ u\in \mc U}} \int  \frac{-\mc L u(x)}{c+u(x)} \alpha(\dd x)
	\end{align*}
	Therefore, for every $k\in \mathds{N}$
	\begin{align*}
	\sup_{\substack{R,c> 0\\ u_1, \ldots, u_k\in \mc U}}\wtilde{\Lambda}(\xi, R, c, u_1, \ldots, u_k)=\sup_{\alpha_1, \ldots, \alpha_k\in\xi}\sum_{i=1}^{k}	I(\alpha_i)\,.
	\end{align*}
	Since
	\begin{align*}
	\wtilde{I}(\xi) =\sup_{k\in\bb N}\sup_{\alpha_1,\ldots, \alpha_k\in\xi} \sum_{i=1}^{k}I(\alpha_i)
	\end{align*}
	we can conclude.
\end{proof}

We now come to the proof of Proposition~\ref{prop:UB}.
\begin{proof}
	Since $\widetilde{\mc X}$ is a compact set, to prove \eqref{eq:Upp} it is enough to prove that if $\xi\in \widetilde{\mc X}$ and $B_{\delta}(\xi)$ is a ball of radius $\delta$ around $\xi$ then
	\begin{align*}\label{eq:UBdelta}
	\limsup_{\delta\to 0}\limsup_{t\to \infty}\dfrac{1}{t}\log
	Q_t(B_\delta(\xi))\leq-\widetilde{I}(\xi).
	\end{align*}
	In fact, let $F\subseteq \widetilde{\mc X}$ be a closed set. Given $\eps>0$ and $\xi\in\widetilde{\mc X}$ there is $\delta(\xi)=\delta(\xi,\eps)$ such that
	\begin{equation*}
	\limsup_{t\to \infty}\dfrac{1}{t}\log
	Q_t(B_{\delta(\xi)}(\xi))\leq-\widetilde{I}(\xi) +\eps\,.
	\end{equation*}
	Writing
	$$F\subset\bigcup_{\xi\in F}B_{\delta(\xi)}(\xi)\,,
	$$
	since $F$ as a closed set in the compact space $\widetilde{\mc X}$ is compact as well
	there exist $\xi_1, \xi_2, \ldots, \xi_n\in F$ such that
	\begin{align*}
	F\subset \bigcup_{i=1}^{n}B_{\delta(\xi_i)}(\xi_i).
	\end{align*}
	Therefore
	\begin{align*}
	\limsup_{t\to \infty}\dfrac{1}{t}\log Q_t(F)&\leq \limsup_{t\to \infty}\dfrac{1}{t}\log\bigg(\sum_{i=1}^{n}Q_t(B_{\delta(\xi_i)}(\xi_i))\bigg)\\
	&=\max_{i=1}^{n}\bigg\{\limsup_{t\to \infty}\dfrac{1}{t}\log Q_t(B_{\delta(\xi_i)}(\xi_i))
	\bigg\}\\
	&\leq\max_{i=1}^{n}\{-\widetilde{I}(\xi_i)\}+\varepsilon\\
	&\leq-\inf_{\xi\in F}\widetilde{I}(\xi)+\varepsilon.
	\end{align*}
	Fix $k\in\bb N$, $u_1, \ldots, u_k\in\mc U$, $c,R>0$ and write $F=F(u_1,\ldots, u_k, c, R, \wtilde\mu)$, where $F(u_1,\ldots, u_k, c, R, \cdot):\widetilde{\mc M}_1(\R^d)\to \R$ is defined in \eqref{eq:F}. Then
	\begin{align*}
	Q_t(B_\delta(\xi))=\mathds{P}(\widetilde{L}_t\in B_\delta(\xi))&\leq \mathds{P}(F(\widetilde{L}_t)\in F(B_\delta(\xi)))\\
	&\leq\mathds{P}(F(\widetilde{L}_t)\geq \inf_{\wtilde{\mu}\in B_\delta(\xi)}F(\wtilde{\mu}))\\
	&\leq \exp \Big(-t\inf_{\wtilde{\mu}\in B_\delta(\xi)}F(\wtilde{\mu})\Big)\mathds{E}\Big[\exp(tF(\widetilde{L}_t))\Big].
	\end{align*}
	From Lemma~\ref{lem:UBfundamental} we have that
	\begin{align*}
	\limsup_{t\to \infty}\dfrac{1}{t}\log Q_t(B_\delta(\xi))&\leq -\inf_{\wtilde{\mu}\in B_\delta(\xi)}F(\wtilde{\mu})+\limsup_{t\to \infty}\dfrac{1}{t}\log\mathds{E}\Big[\exp(tF(\widetilde{L}_t))\Big]\\
	&\leq -\inf_{\wtilde{\mu}\in B_\delta(\xi)}F(\wtilde{\mu})\,.
	\end{align*}
	%
	Then, by Lemma~\ref{lem:Ftildelower}   for all $k\in\bb N$, $u_1,\ldots, u_k\in \mc U$, $c,R>0$
	\begin{align*}
	\limsup_{\delta\to 0}\limsup_{t\to \infty}\dfrac{1}{t}\log Q_t(B_\delta(\xi))\leq-\Lambda(u_1, \ldots, u_k, C, R, \xi)\,.
	\end{align*}
	Hence, by Lemma~\ref{lem:ratefct}
	\begin{align*}
	\limsup_{\delta\to 0}\limsup_{t\to \infty}\dfrac{1}{t}\log Q_t(B_\delta(\xi))&\leq \inf_{\substack{R, c>0, \ k\in \mathds{N}\\ u_1, \ldots, u_k}} (-\Lambda(u_, \ldots, u_k, c, R, \xi))\\
	&=-\sup_{\substack{R, c>0, \ k\in \mathds{N}\\ u_1, \ldots, u_k}} \Lambda(u_, \ldots, u_k, c, R, \xi)\\
	&=-\widetilde{I}(\xi)
	\end{align*}
	and the result is proved.
\end{proof}
\subsection{Lower semicontinuity}
\begin{lemma}
	If $\xi_n\to\xi$ in $\widetilde{\mc X}$, then
	\begin{align*}
	\liminf_{n\to\infty}\widetilde{I}(\xi_n)\geq\widetilde{I}(\xi).
	\end{align*}
\end{lemma}
\begin{proof}
	The proof of this lemma follows largely along the lines of the proof of Lemma~\ref{lem:Ftildelower}. In particular, we will use the identification
	\begin{align*}
	\wtilde{I}(\xi)=	\sup_{\substack{R, c>0, \ k\in \mathds{N}\\ u_1, \ldots, u_k\in U}}\wtilde{\Lambda}(\xi,R,c, u_1, \ldots, u_k)
	\end{align*}
	established in Lemma~\ref{lem:ratefct}.
	Therefore, we fix $R,c>0$ and $u_1,\ldots, u_k\in\mc U$ and we define $u_{i,R}$ as in the formulation of Lemma~\ref{lem:Ftildelower}.

	Let $(\wtilde{\mu}_{n})$ be a sequence in $\widetilde{\mc X}$, converging to $\xi =\{\wtilde{\alpha}_i\}_{i\in I}\in\widetilde{\mc X}$. We suppose that $\wtilde{I}(\mu_{n})\leq l$ for some $0\leq l<\infty$ and we will prove that $\wtilde{I}(\xi)\leq l$.

	Initially consider the case where for each $n$ one has that  $\wtilde{\mu}_{n}$ consists of a single orbit, so that $\wtilde{I}(\mu_{n})=I(\mu_{n})$.
	Restricting to a subsequence if necessary, for some $k\geq 1$ we can write by Corollary~\ref{cor:decomposition}
	\begin{align*}
	\mu_n=\sum_{i=1}^{k}\alpha_{n,i}+\beta_n
	\end{align*}
	so that
	\begin{itemize}
		\item $\alpha_{n,i}, \  i=1,\ldots,\ k,$ and $\beta_n$ are sequences of sub-probability measures in $\Z^d$;
		\item for each $i=1, \ldots, k$ there are sequences $a_{n,i}, \ i=1,\ldots, k,$ in $\Z^d$ that can be chosen such that
		\begin{align*}
		&\alpha_{n,i}\ast\delta_{a_{n,i}}\Rightarrow\alpha_i\in \wtilde{\alpha}_i, \ n\to \infty\\
		&\lim_{n\to \infty}\min_{i\neq1} \vert a_{n,i}-a_{n,j}\vert=\infty
		\end{align*}
		and $(\beta_n)$ is widely separated from each $(\alpha_{n,i})$.
		\item 	The supports of $\alpha_{n,1}, \alpha_{n,2}, \ldots, \beta_n$ are all disjoint and for each $i$ there exists a sequence $ (R_{n,i})_n$ tending to infinity such that
		$$\Supp (\alpha_{n,i})\subset \mathcal{B}(-a_{n,i}, R_{n,i})$$
		and
		$$\Supp(\beta_{n,i})\subset \big[\cup_i \mathcal{B}(-a_{n,i}, R_{n,i})\big]^\complement.$$
	\end{itemize}

	Next, we note that by the compactness of the support of $u_{i,R}$ the support of each $u_{i,R}(\cdot-a_{n,i})$ is contained in $B(-a_{n,i}, R)$. We see that we are in the same setting as in the proof of Lemma~\ref{lem:Ftildelower}. Therefore, as in that proof we can conclude that
	%
	\begin{align*}
	l\geq\liminf_{n\to\infty} I(\mu_{n})\geq\liminf_{n\to \infty}\sum_{i=1}^k\int \dfrac{\mathcal{-L}u_{i,R}(\cdot-a_{n,i})}{u_{i,R}(\cdot-a_{n,i})}\dd\mu_{n}
	=\sum_{i=1}^{k}\int\dfrac{\mathcal{-L}u_{i,R}}{u_{i,R}}\dd\alpha_{i}
	=\sum_{i=1}^{k}I(\alpha_{i})\,.
	\end{align*}
	As in the proof of Lemma~\ref{lem:Ftildelower} we can obtain the shift by $b$ by simply shifting each of the $a_{n,i}$'s.
	Taking on both sides the supremum over all $R,c>0$, $k\in\bb N$ and $u_1,\ldots, u_k$ allows to conclude the cases where each $\wtilde{\mu}_n$ consists of a single orbit.
	Finally,
	to treat the general case, that is when $\wtilde{\mu}_n$ has possibly more than one orbit, the idea is
	the same as in the last paragraph of~\cite[proof of Lemma~4.2]{MV2016}.
\end{proof}

\section{Applications of Theorem~\ref{thm:main}}\label{sec:applications}
In this section we present two applications of Theorem~\ref{thm:main}.
\\

\subsection{Application 1} We consider the same assumptions as in Section~\ref{sec:main} and we define the function
\begin{equation}
V(x,y)= \one_{\{x=y\}}
\end{equation}
on $\Z^d\times\Z^d$. For any element $\xi\in\wtilde{\mc X}$ we then consider
\begin{equation}
\Lambda(V,\xi)=\sum_{\alpha\in\xi} \int V(x,y)\, \alpha(dx)\alpha(dy)\,,
\end{equation}
which is continuous in $\xi$ be the definition of the metric $\textbf{D}$ on $\wtilde{\mc X}$. Note that
\begin{equation}
\Lambda(V, L_t)= \frac{1}{t^2} \int_0^t\, ds\,\int_0^t\, du\, \one_{\{X_s=X_u\}}\,,
\end{equation}
which is simply the (rescaled by $t^2$) intersection local time of the random walk $X$.
To continue define a new rate function $I'$ on $\bb R$ via
\begin{equation}
I'(y)= \inf\{\wtilde I(\xi):\,\xi\in\wtilde{\mc X}, \Lambda(V,\xi)=y\}\,.
\end{equation}
Theorem~\ref{thm:main} together with the contraction principle then imply the following result.
\begin{theorem}
The process of rescaled intersection local times $(\Lambda(V, L_t))_{t\geq 0}$ satisfies a strong large deviation principle with rate function $I'$ and rate $t$.
\end{theorem}

\noindent
\subsection{Application 2} Consider a bounded non-negative function $V\in\mc F_2$ which is non-zero, taking its maximum in $(0,0)$ such that $V(0,0)>\sum_{y}a_{0,y}$, and define
\begin{equation}
	Z_t= \bb E_0\Big[\exp\Big(t\int V(x,y) L_t(dx)L_t(dy)\Big)\Big]\,.
\end{equation}
\begin{remark}\rm
The definition of $Z_t$ is somewhat related to the Polaron problem~\cite{DV83} which corresponds to the choice $d=3$,  $V(x,y)=\tfrac{1}{|x-y|}$ and where the underlying process is a Brownian motion instead of a random walk. We see that this function does not fall into the class of functions considered here. The reason is simply that this choice of function renders $Z_t$ meaningless. One could possibly circumvent this by studying an appropriate choice of approximations of the potential of the Polaron problem. The drawback would be that this would translate to sequences of empirical measures that depend on an additional approximation parameter $N$. In this work we prefer to refrain from it.
\end{remark}
It follows from Laplace-Varadhan's~\cite[Theorem 4.3.1]{DemboZei10:book} Lemma that
\begin{equation}\label{eq:varfor}
	\lim_{t\to\infty}\frac1t \log Z_t
	=\sup_{\xi\in\wtilde{\mc X}}\big\{\Lambda (V,\xi)-\wtilde{I}(\xi)\big\}=:\lambda\,.
\end{equation}
We will analyze the variational formula in the sequel and will show that there is a maximizer and that this maximizer is a probability measure. We will start with the existence of the maximizer.
Consider a sequence $(\xi_n)_n$ in $\wtilde{\mc X}$ such that
\begin{equation*}
\lim_{n\to\infty}\big\{\Lambda (V,\xi_n)-\wtilde{I}(\xi_n)\big\}=\lambda\,.
\end{equation*}
By the compactness of $\wtilde{\mc X}$ we can extract a subsequence of $(\xi_n)_n$ that converges to an element $\xi\in\wtilde{\mc X}$. For ease of notation we denote the subsequence again by $(\xi_n)_n$. Since $\wtilde{I}$ is lower-semicontinuous and $\xi\mapsto \Lambda(V,\xi)$ is continuous we have that
\begin{equation*}
\lambda=\limsup_{n\to\infty}\big\{\Lambda (V,\xi_n)-\wtilde{I}(\xi_n)\big\}\leq \Lambda(V,\xi)-\wtilde{I}(\xi)\,,
\end{equation*}
from which we can conclude that $\xi$ is a maximizer.
We next claim that the maximizer has only one component. To see that assume that the maximizer $\xi$ has at least two components, denoted by $\alpha_1$ and $\alpha_2$. Consider representants $\alpha_1$ and $\alpha_2$ such that $\int V(x,y)\alpha_1(dx)\, \alpha_2(dy)>0$. This is possible since $V$ is non-zero. We then define $\alpha_3=\alpha_1+\alpha_2$ and
\begin{equation*}
	\wtilde\xi=\{\alpha_3, \xi\setminus\{\alpha_1, \alpha_2\}\}\,.
\end{equation*}
Then,
\begin{equation*}
\begin{aligned}
	\Lambda(V, \alpha_3) &= \int V(x,y)\alpha_1(dx)\, \alpha_1(dy) +\int V(x,y)\alpha_2(dx)\, \alpha_2(dy) + 2\int V(x,y)\alpha_1(dx)\, \alpha_2(dy)\\
	&> \Lambda(V,\alpha_1)+\Lambda(V,\alpha_2)\,,
	\end{aligned}
\end{equation*}
by the choice of $\alpha_1$ and $\alpha_2$.
By Lemma~\ref{lem:PropertiesI}
\begin{equation*}
	I(\alpha_3)\leq I(\alpha_1)+ I(\alpha_2)\,.
\end{equation*}
From these two observation it follows that
\begin{equation*}
\Lambda (V,\xi)-\wtilde{I}(\xi) < \Lambda (V,\tilde\xi)-\wtilde{I}(\tilde\xi)\,,
\end{equation*}
from which the claim follows.
We next show that any maximizer is a probability measure.
To that end assume that $\beta$ is a sub-probability measure whose orbit maximizes the variational formula in~\eqref{eq:varfor}. Define $p=\beta(\Z^d)\in (0,1)$ and consider the probability measure $\mu=\tfrac{\beta}{p}$ (it will follow from the arguments further below that $p>0$).
Then,
\begin{equation*}
\Lambda(V, \mu)-I(\mu)= \frac{1}{p}\Big(\frac{\Lambda(V,\beta)}{p}-I(\beta)\Big)
\geq \frac{1}{p}(\Lambda(V,\beta)-I(\beta)) > \Lambda(V,\beta)-I(\beta)\,,
\end{equation*}
contradicting the assumption that $\beta$ maximizes~\eqref{eq:varfor}. Here we used that the supremum must be positive. Indeed, define
\begin{equation*}
\tau_1=\inf\{t\geq 0:\, X_t\neq 0\}\,,
\end{equation*}
which is the first jump time of the random walk $X$ and has an exponential distribution with parameter $\sum_{y\in\Z^d}a_{0,y}$. It then follows that
\begin{equation*}
Z_t\geq \bb E_0\Big[\exp\Big(t\int V(x,y) L_t(dx)L_t(dy)\Big)\one_{\{\tau_1>t\}}\Big]\geq \exp\Bigg(t\Big(V(0,0)-\sum_{y\in\Z^d}a_{0,y}\Big)\Bigg)\,,
\end{equation*}
and the claim follows by the assumption that $V(0,0)>\sum_{y\in\Z^d}a_{0,y}$.
The above analysis implies the following result:
\begin{theorem}
Consider a function $V:\Z^d\times\Z^d\to\R_+$ satisfying the above assumptions, then
\begin{equation*}
\lim_{t\to\infty}\frac1t\log \bb E_0\Big[\exp\Big(\frac1t\int_0^t\,ds\, \int_0^t\, du\, V(X_s-X_u)\Big)\Big] =\sup_{\mu\in\mc M_1(\Z^d)}\big\{\Lambda (V,\xi)-I(\xi)\big\}\,.
\end{equation*}
The above variational formula has at least one maximizer. Moreover, if $\mu$ is a maximizer, then for any $x\in\Z^d$, the shifted measure $\mu_x:=\mu*\delta_x$ is a maximizer as well.
\end{theorem}
\begin{remark}\rm
It would be interesting to know under which conditions the above variational formula has a unique (unique modulo translations in space) maximizer. This however seems a delicate problem. Even in the case in which $\Z^d$ is replaced by $\R^d$ this is not a trivial issue, see~\cite{Lieb77}. We do not pursue this question here.
\end{remark}
\section*{Acknowledgements}
T.F.\ was supported by by the National Council for Scientific and Technological Development - CNPq via a Bolsa de Produtividade 311894/2021-6.
 D.E. was supported by the National Council for Scientific and Technological Development - CNPq via a Bolsa de Produtividade 303348/2022-4 and Serrapilheira Institute (Grant Number Serra-R-2011-37582). D.E. and T.F. acknowledge support by the National Council for Scientific and Technological Development - CNPq (Universal Grant 406001/2021-9). J.S.\ acknowledges support by Capes through a PhD scholarship.

\bibliography{bibliography}
\bibliographystyle{plain}

\end{document}